\documentclass{amsart}

\usepackage[toc,page]{appendix}

\usepackage{amsaddr}

\usepackage{amsmath, amssymb, amsthm}
\usepackage{graphicx}
\newtheorem{definition}{Definition}
\newtheorem{theorem}{Theorem}
\newtheorem{lemma}{Lemma}
\usepackage{algorithm}
\usepackage{algpseudocode}
\usepackage{subcaption}

\usepackage{tikz}

\newcommand{\legendentry}[2]{%
  \begin{tikzpicture}[baseline=(legendnode.base)]
    \fill[#1] (0,0) rectangle (0.3,0.3);
    \node[right=0.1cm] (legendnode) at (0.3,0.15) {#2};
  \end{tikzpicture}%
}


\title{Curvature of high-dimensional data}

\author{Jiayi Chen$^1$}

\author{Mohammad Javad Latifi Jebelli$^1$}

\author{Daniel N. Rockmore$^{1,2,3}$}

\address{$^1\;$ Department of Mathematics, Dartmouth College, Hanover, NH 03755}

\address{$^2\;$Department of Computer Science, Dartmouth College, Hanover, NH 03755}
\address{$^3\;$The Santa Fe Institute, Santa Fe, NM 87501}

\begin{document}



\begin{abstract}
We consider the problem of estimating curvature  where the data can be viewed as a noisy sample from an underlying manifold. For manifolds of dimension greater than one there are multiple definitions of local curvature, each suggesting a different estimation process for a given data set. Recently, there has been progress in proving that estimates of ``local point cloud curvature" converge to the related smooth notion of local curvature as the density of the point cloud approaches infinity.  Herein we investigate practical limitations of such convergence theorems and discuss the significant impact of bias in such estimates as reported in recent literature. We provide theoretical arguments for the fact that bias increases drastically in higher dimensions, so much so that in high dimensions, the probability that a naive curvature estimate lies in a small interval near the true curvature could be near zero.  We present a probabilistic framework that enables the construction of more accurate estimators of curvature for arbitrary noise models.  The efficacy of our technique is supported with experiments on spheres of dimension as large as twelve.
\end{abstract}

\maketitle

The ``manifold hypothesis" assumes that a data cloud can be viewed as noisy samples of an underlying manifold. Examples include digital images or word embeddings from a large language model. In this context, one might attempt to estimate the manifold curvature by estimating tangent spaces at nearby points and comparing them. We show here that such curvature measurements are highly unreliable for high-dimensional data sets. Furthermore, we rigorously study the origin of this bias and seek a remedy for these shortcomings by recovering the correct curvature value through a combination of ideas from geometry and probability.
    
Curvature estimation is a part of ``data geometry", a subject of central interest in today’s world of data science 
where high-dimensional data, by which we mean feature vectors of tens, hundreds, or even thousands of dimensions are now common. The manifold hypothesis presents natural estimation questions regarding the geometric properties of the underlying  manifold, such as local dimensionality and local curvature (intrinsic or extrinsic), which are in turn viewed as properties of the data. 
    
Some standard approaches to estimating curvature for data of dimension greater than three are biased \cite{hickok2023,RiemannianCurvaturePNAS} (see below).  Our primary contribution is to remedy the bias by approaching the problem from a probabilistic point of view. By this, we mean given a noise model for the data (as noisy samples from an underlying manifold), we aim to compute the distribution of the now noisy estimate of curvature as derived from the sample. This ``pushforward" of the original noise distribution encodes detailed information about the empirical estimate. In particular, while the pushforward may converge (in sampling density) to the true curvature, it can include a form of bias that can distort -- possibly radically -- any finite estimate. 
    
    Working in the context of {\em absolute variation curvature}, a notion of curvature that measures local variation in tangent space estimates near a point (see Definition 2), we derive a closed form for the pushforward of a  von Mises-Fisher-based (vMF-based) noise model for the underlying random sample (Theorems 2) of curvature. This extends to mixtures of vMF distributions (Theorem 3).  The closed form enables us to address the bias in the pushforward in the form of a much more accurate curvature estimator in a general setting.  Since any noise model can be well-approximated by a finite mixture of vMF distributions, our result applies more generally (Theorem 4) and suggests a technique for mitigating bias in other curvature calculations that depend on local estimates of tangent space (and the second fundamental form). We demonstrate our approach in a series of experiments on spheres  $S^3,S^5,S^{10},$ and $S^{12}$ with radii of one and two. Although the form of our results depends on working with a codimension one embedding, our technique suggests a natural extension to the setting of arbitrary codimension.

    \parshape=0 
    
\subsection*{Background and Related Work}

The data science setting requires that there be some kind of initial data. In some cases the data is indirect, in the sense that all that one may have is a similarity or dissimilarity matrix, the latter of which can be transformed to a distance matrix. A more standard setting is explicit data in the form of feature vectors or coordinates, and thus, an embedding of the measured phenomenon of interest, whose geometric structure awaits discovery. The wealth of applications to image and volume analysis have motivated a good deal of study of curvature in the contexts of dimensions two and three. Our interest is in arbitrary -- and higher -- dimensions. 

Our work is situated in the growing body of work that aims at understanding local geometric structure in high-dimensional data sets.  To that, first order structure is equivalent to the determination of local dimensionality and is achieved via various approaches to local PCA (see e.g., \cite{LocalDim2-10.1145/2783258.2783405, LocalDim1-1671801,LocalDim3-10.1145/1102351.1102388}). This is effectively the estimation of the tangent space at a point. Local curvature gets at second order information. In higher dimensions there is no unique measure of curvature. First off it is of either an {\em extrinsic} or {\em intrinsic} character. They can be related computationally, but they are different.

   {\em  Extrinsic curvature} gives a measure of local deviation of the tangent space, in essence asking how the orientation of the tangent space changes near a point. This is effectively giving a local second order description (the tangent space is a first order description) of the manifold. This definition requires an embedding of the manifold in some ambient Euclidean space and is the setting for most of data science where we start with high-dimensional data, an implicit embedding. Our measure of absolute variation curvature is extrinsic (cf. Section~\ref{sec:curv}).  The local second order description encoded in curvature relates this work to the problem of finding good local second order fits of data that knit  together the local first order fits. This problem is addressed in \cite{GilbertONeill}.  More generally, the {\em second fundamental form} is a tensor that encodes all of the second order local behavior and as we will show (cf. Section~\ref{sec:Discussion}), our work can extended to this more complete description.  
    
{\em Intrinsic curvature} is a notion of curvature that is independent of embedding and solely relies on the metric structure.  A detailed explanation of this topic is beyond the scope of this paper, but there are many good sources for those interested in a thorough understanding (see e.g., \cite{Lee}). One method  is manifested  in a consideration of how the volume of a ball (restricted to the manifold) centered at a point of interest changes as the radius grows. This can even depend on the direction of growth. For example, {\em scalar curvature}  at a point is a measure  recording a comparison of the volumes of a ball of a given radius confined to the manifold (and thus determined using the metric of the manifold) and the ball of the same radius in the Euclidean space of the same dimension (see e.g., \cite{Ache2014RicciCA}). Positive curvature indicates less volume and negative curvature indicates greater volume. (In two dimensions, compare the areas of disks of a given radius confined to a sphere, a plane, and a saddle.)

The papers \cite{hickok2023} and \cite{RiemannianCurvaturePNAS}  give two approaches to the calculation of scalar curvature.   In \cite{hickok2023}  this is accomplished   using only a distance matrix as the original data. They show that their method recovers scalar curvature in the limit with increasing density of samples. They also provide experiments recovering known scalar curvature from noisy samples of spheres in a range of dimensions. The approach detailed in \cite{RiemannianCurvaturePNAS} connects the computation of scalar curvature to a calculation of extrinsic curvature. Included there are many numerical experiments, both to data where there is some ground truth (spheres of a range of dimensions, a Klein bottle, as well as image patch data) as well as a  cellular RNA expression data. 

{\em Ricci curvature} is another intrinsic curvature measure derived from the relative (with respect to Euclidean space) asymptotic growth of balls centered on the point of interest. The papers  \cite{Ache2014RicciCA} and \cite{trillos2024} give algorithms for its estimation from data sampled from the underlying manifold. The latter relates it  {\em Ollivier-Ricci curvature} which is derived from a measure comparing nearby random walks on the manifold \cite{OLLIVIER2009810}. This is in turn related to {\em diffusion curvature} \cite{DiffusionCurvatureNIPS}  which defines curvature (at a given point $x$) in terms of the probability that a random walk stays within a given volume after a fixed number of steps, a formalism that comes from the foundational work that introduced the notion of {\em diffusion maps} to high-dimensional data analysis \cite{coifman2006diffusion}.

    The authors in \cite{hickok2023} and \cite{trillos2024} provide convergence analysis to show that in the limit of infinite sampling density they can recover the continuous notions of scalar and Ricci curvature respectively. Also, see \cite{Aamari2019}, \cite{OLLIVIER2009810} and \cite{Zhang2025}.  \\

\subsection*{Bias reports in the literature} Convergence in probability (and in the presence of noise) does not guarantee anything about the bias in such curvature estimations (in the presence of noise) and despite remarkable progress, it is not clear if the curvature value computed using a finite density of points is a good representative for the true curvature. In particular, we find in the current literature examples of systematic bias in the estimates, acknowledged by the researchers, but unaddressed. 

First, consider \cite{RiemannianCurvaturePNAS}. Therein,  Figure S2 of their associated Supporting Information displays the results of estimating scalar curvature  using uniformly distributed points on $S^2$ embedded in Euclidean spaces over a range of dimensions and convolved with Gaussian noise over a range of standard deviations. We find in these figures a bias in curvature estimation that increases  with variance in the noise model and the dimension of the embedding.\footnote{In one set of experiments, the embedding is fixed at $\mathbb{R}^5$ while $10k$ points are sampled for Gaussian noise with $\sigma \in \{0.001,0.003,0.01,.03, 0.1, 0.3\}$, while in another set, the ambient dimension ranges over $\{3,10,20,30,40,60, 80, 100\}$ for $\sigma\in\{0.01, 0.03, 0.05\}$.}. In the words of the authors, ``The nonzero value of the mean error indicates that our estimator is biased," which they attribute to the dependence on a nonlinear function of the values comprising the second fundamental form (cf. \cite{RiemannianCurvaturePNAS}, Eq. (5)). 

    As a second example, we consider \cite{hickok2023}. Therein, we see in Figure 5, highly biased estimates of scalar curvature, using four different numerical techniques kernel estimation from $10k$ uniformly distributed points from $S^2 \subset \mathbb{R}^3$, convolved with Gaussian noise for $\sigma \in \{0.001,0.003,0.01, 0.03\}$. 

We argue in this paper that the bias is intrinsic to the methods of estimation and show, using a probabilistic framework and approach, a way to mitigate this bias. This claim is supported by extensive numerical experiments on high-dimensional spheres.

\subsection*{Organization} 
We begin with a general formulation of the problem, which is to recover a measure of local (extrinsic) curvature, estimated from noisy data sampled from a manifold $\mathcal{M} \subset \mathbb{R}^n$ of dimension $k < n$. While there are various notions of curvature (e.g., diffusion curvature, mean curvature, etc.). We set ourselves the goal of estimating a quantity we call the {\em absolute variation curvature}, $\omega$, which measures the variation in relative tangent space orientation near a given point (cf., Section~\ref{sec:curv}). In the general setting (presented in Section \ref{sec:prob_calib}) this leads one to consider a PCA-based tangent space approximation as a Grassmannian-valued random variable, $X$ and compute the push-forward of this measure under the curvature computation. In the case of co-dimension one, the situation simplifies nicely as any tangent space approximation is equivalent to the estimation of a normal vector. 

Section~\ref{sec:prob_decoding_avc} is devoted to the detailed study of the probabilistic framework as applied to absolute variation curvature. We find an explicit expression for the distribution of absolute variation curvature, $\Omega$, a random variable that contains the true curvature $\omega$ as a distribution parameter. In Section~\ref{sec:experiments} we use the theoretical results of Section 4 to develop a maximum likelihood method to estimate $\omega$ from random observations of $\Omega$ and provide numerical experiments. Here, $\Omega$ is the naive (biased) curvature obtained from curvature formula and $\omega$ is the true curvature taking into account the effect of noisy samples. Our techniques avoids the inherent bias that has been observed by others that increases with dimensionality. While we are working in the specific context of the  absolute variation curvature, it, like other curvature measures can be related back to the second fundamental form and the shape operator so that in principle our general framework can be applied to other ``derivative" forms of curvature.  This, and suggestions for future work are addressed in our closing Discussion (Section~\ref{sec:Discussion}).  \\

\section{Probabilistic Decoding of Curvature -- The General Case}\label{sec:prob_calib}

\subsection{The Abstract Framework}
 We demonstrate the general idea of our approach in an abstract setting that in principle  can be applied  to any curvature estimation method. We later specialize this to tangent space techniques. Given $\mathcal{M} \subset \mathbb{R}^n$ let $\mathcal{X}(\mathcal{M})\subset \mathbb{R}^n$  denote a finite set of points that manifest a noisy sample of points about $\mathcal{M}$. One can either think of this as generated by a random process (given $\mathcal{M}$) or as a point cloud with an assumed underlying manifold structure (i.e., as satisfying the manifold hypothesis). We write $\mathcal{X}$ for $\mathcal{X}(\mathcal{M})$ when $\mathcal{M}$ is understood. For fixed $x \in\mathcal{M}\subset \mathbb{R}^n$, let $\Omega$ denote a general expression that estimates an explicit notion of curvature at $x$ using points in $\mathcal{X}$ ($\Omega$ is a function of in a neighborhood of $x$ in $\mathcal{X}$ but we omit this dependence to simplify the notation). We refer to $\Omega$ as the \textbf{naive curvature} at $x$. We use the term ``naive" for such curvature estimation because it ignores the effect of density and noise of points in $\mathcal{X}$ on the accuracy of the estimation. A crucial observation is that in general, for a point cloud assumed to be sampled from $\mathcal{M}$ with ambient noise, $\Omega$ might be a biased  estimator for curvature of the underlying manifold. 

We argue that one should consider the expression of $\Omega$  as a random variable depending on a random realization of $\mathcal{X}$. Further, we assume  $\Omega ~\sim \mathcal{R}(\eta, \omega)$  depends on two independent parameters $\eta$ and $\omega$. The first (possibly vector) parameter $\eta$ depends on the random process that models the distribution of the point cloud $\mathcal{X}$. The  parameter $\omega$ is  the true curvature value at $x$. The value of $\omega$ can be estimated from some noisy samples of $\Omega$ (this formulates the estimate of $\omega$ as the solution to an inverse problem). We  call this $\omega$ the \textbf{decoded curvature estimate}. 

Hence, the big picture of our estimation framework is as follows:  (1)  Incorporate  prior knowledge of point cloud density and ambient noise to estimate $\eta$;  (2) Use samples of $\Omega$ (the naive curvature calculations) to estimate the only remaining parameter $\omega$. Steps (1) and (2) then produce a distribution for $\mathcal{R}(\eta, \omega)$ from which one can derive an estimate for $\omega$. This is completely general. We now   provide a detailed explanation for curvature estimation techniques that rely on tangent spaces.   \\

\noindent \textbf{Naive curvature using tangent space methods.} Consider a manifold $\mathcal{M}$ embedded in $\mathbb{R}^n$. A notion of curvature quantifies the variation of tangent spaces   in an infinitesimal neighborhood of a fixed base point. Given a point cloud of data $\mathcal{X}$ produced as noisy samples from (and near)  $\mathcal{M}$, and given a fixed point $x \in \mathcal{M}$, we consider a collection of points $x_1, \dots, x_q \in \mathcal{X} \subset \mathbb{R}^n$ in a small neighborhood of $x$. For each $x_j$, one can compute an approximation to the tangent space at $x_j$ using a local regression or PCA technique. An estimate for the curvature at $x$ is then given as an expression in terms of all these tangent space estimates. For instance, one could choose a particular direction on the manifold and observe how the tangent space changes as we step in that direction. Another approach is to measure the angles between the  tangent space at the center point $x$ with each of these nearby tangent spaces. This will lead to the notion of absolute variation curvature as discussed in the next section.  \\


\noindent \textbf{Decoded curvature for tangent space methods}. We treat the tangent space approximation as a Grassmannian valued random-variable. (The Grassmanian is the vector space of $m$-dimensional subspaces -- for some fixed $m,n$ -- of an $n$-dimensional vectors space.) As a result, the computational steps to arrive at a value for curvature correspond to transformations applied to these random variables. The result is a random variable which encodes the true curvature as a parameter. Let us go over this construction step by step. \\

    \textbf{Step 1.  Identify a probability distribution for tangent space estimations}. Let $X_i$  represent the tangent space estimation at $x_i$   treated as a Grassmannian valued random variable. In the case of manifolds of codimension one in $\mathbb{R}^n$, $X_i$ is a random normal vector on the unit sphere. A natural model is to assume that $X_i  \sim \mathcal{T}(\kappa)$, a von-Mises-Fisher distribution (vMF) with parameter $\kappa$. The choice of parameter depends on the density of the points in $\mathcal{X}$  and the intensity of ambient noise, and the value of $\kappa$ is calculated prior to any curvature estimation. For a fixed density and noise level we set up numerical experiments to estimate the correct parameter value for $\kappa$. One can also employ a theoretical framework to establish the true probability distribution for $X_i$ but in this paper we use mixtures of  von-Mises-Fisher distribution to approximate the true distribution of $X_i$.   \\
    
    \textbf{Step 2. Compute the push-forward of tangent random variables under naive curvature formula}.  In  the naive curvature calculation we would take  random variables $X_1, \dots, X_q$ as input and produce a random variable $\Omega$ (that depends on the tangent space estimates at each point) as output. We now need to calculate the pushforward of the input random variables to obtain an explicit family of probability distributions that accounts for a depends on both $\eta$ (this is $\kappa$ in the vMF case) and the true curvature  $\omega$ as parameters. In general the calculation of the pushforward can be difficult (even if it is always numerically approachable), but in the case of absolute variation curvature and a vMF noise  model we are able to give an explicit form for the push-forward (cf., Theorem~\ref{thm:curv_pushforward}).   \\

\textbf{Step 3. Obtain samples for $\Omega$ and estimate $\omega$}. This last step is effectively an inverse problem. Given that naive curvature calculations provide samples for random variable $\Omega$, we can solve the inverse problem of finding $\omega$ that best explains these samples. Given the fact that we know the probability distribution of $\Omega \sim \mathcal{R}(\kappa, \omega)$ this is simply a maximum likelihood estimation problem for recovering $\omega$ (recall that $\kappa$ is already known). \\

The main technical difficulty in executing above idea is to provide an expression for  the distribution for $\mathcal{R}(\kappa, \omega )$. For instance, such a distribution is highly non-trivial in the case of diffusion curvature (see \cite{DiffusionCurvatureNIPS}). In this manuscript, we will work in the context of absolute variation curvature (cf. Section~\ref{sec:curv}) for which we can provide an explicit expression for the corresponding family of distributions (cf. Theorem~\ref{thm:curv_pushforward}). \\

\section{Curvature}\label{sec:curv}

One common way to define the curvature of a differentiable curve is to use the reciprocal of the radius of the osculating circle that best approximates the curve at a point on the curve. We start with the simple case of $\mathcal{M} = S^m_r \subset \mathbb{R}^{m+1}$, the $m$-dimensional sphere of radius $r>0$: given two points $x,y \in \mathcal{M} = S^m_r \subset \mathbb{R}^{m+1}$, we note
$$
\frac{1}{r} = \frac{2 \sin (\frac{\theta}{2})}{\| x - y\|_{\mathbb{R}^{m+1}}}
$$
where $\theta$ is the angle between the tangent spaces $T_x \mathcal{M}$ and $T_{y} \mathcal{M}$, and $\|\cdot \|_{\mathbb{R}^{m+1}}$ is the Euclidean distance. Thus, more generally for manifolds embedded in $\mathbb{R}^n$, this suggests that curvature can be expressed as ``change" in tangent spaces as we move the base point. For a general submanifold of dimension $m<n$, the tangent spaces are identified with points on a Grassmannian $Gr(m,n)$. Based on these observations, we define a discrete notion of curvature as follows.  \\

\begin{definition}
    Let $ \mathcal{M}\subset \mathbb{R}^n$ be a smooth manifold of dimension $m$. For $x , y \in  \mathcal{M}$ in a local neighborhood, we define the discrete variation curvature at $x$ in the direction of $y$ by 
    $$\Omega_y(x) = \frac{ 2 \sin(\frac{\theta (x,y)}{2})}{\| x - y\|_{\mathbb{R}^n}}$$
where $\theta (x,y)$ 
is the angle between tangent spaces at $x$ and $y$. Alternatively, we can think of $\theta$ in terms of a distance function defined on the Grassmannian $Gr(m,n)$ of $m$-dimensional linear subspaces of the $n$-dimensional space. 
\end{definition}

In the case of manifolds with codimension one, the value of $\Omega_y(x)$ can be expressed in terms of the unit normal vectors at $x$ and $y$, $\textbf{n}(x)$ and $\textbf{n}(y)$ via 
 $$
 \Omega_y(x)= \frac{2 \sin (\arccos(\textbf{n}(x)\cdot \textbf{n}(y))/2)}{|| x-y ||}.
 $$

We are interested in the case of curvature estimation from a noisy sample of points on the manifold. To that end, we work in a probabilistic framework. Let $x \in \mathcal{M}$ be a point in $\mathcal{X}$ with the tangent space $T_x \mathcal{M}$ and let $y \in \mathcal{X}$ be a nearby point. We treat the estimated tangent space at $y$ as a random variable $X$ that depends on the random realization of point cloud $\mathcal{X}$ (for now we ignore the randomness in estimating $T_{x}\mathcal{M}$). Hence, for a general submanifold of dimension $m$ in $\mathbb{R}^n$, $T_{y}\mathcal{M}$ is a random point in the Grassmannian $Gr(m,n)$. \\

Let $ \mathcal{M}$ be an embedded submanifold of $\mathbb{R}^{m+1}$ of codimension 1, i.e. $ dim (\mathcal{M}) = m$. The absolute variation curvature is defined as follows,

\begin{definition}
 The \textbf{absolute variation curvature} $\omega = \omega(x)$, at $x\in  \mathcal{M}$ is defined as the mean value of $\Omega_y(x)$ when the average is taken over all $y$'s in an infinitesimal sphere around $x$. More precisely, 
    \begin{equation}\label{eq:absvc}
        \omega  = \lim_{\epsilon \rightarrow 0 } \frac{1}{|A_{\epsilon}|} \int_{A_{\epsilon}}\Omega_y(x)\, d\lambda(y), \quad \quad A_{\epsilon} = \partial B_\epsilon(x)
    \end{equation}
    
\end{definition}

Intuitively, this takes into account the average changes of tangent spaces near $x$, for any given tangent directions. See Figure~\ref{fig:curve_tangents} for an illustration of the situation for the unit circle in $\mathbb{R}^2$. In the SI, we show how to express $\omega$ in terms of more familiar geometric quantities such as shape operator and second fundamental form ($\omega$ depends on the eigenvalues of the shape operator, $\mathcal{L}_x$). In principle, the techniques presented in this manuscript can be used to estimate various components of the second fundamental form, though we will focus specifically on $\omega$. \\

The absolute variation curvature depends on the embedding of $\mathcal{M}$ in the Euclidean space. In the context of data science, one can argue that extrinsic notions of curvature, such as $\omega $, are significant. For instance, a noisy point cloud produced by a cylinder in $\mathbb{R}^3$ exhibits an apparent curvature that is captured by $\omega $, despite the surface being intrinsically flat. 


\setlength{\intextsep}{2pt}
\setlength{\textfloatsep}{2pt}
\captionsetup[subfigure]{aboveskip=0pt,belowskip=0pt}

\begin{figure}[H] 
 \begin{subfigure}[b]{0.5\textwidth}
    \centering
    \includegraphics[width=\textwidth]{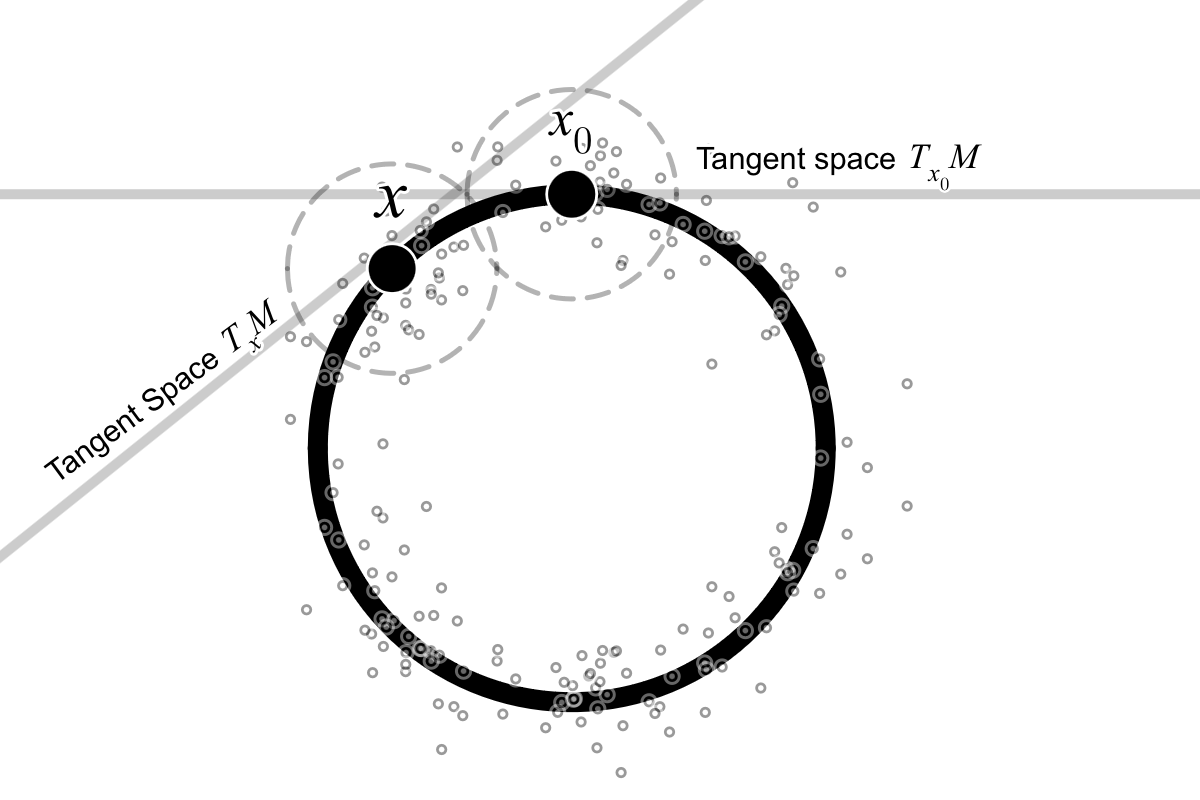}
    \caption{Tangent spaces of two nearby points on a circle. The large circle represent the manifold. The angle $\theta$ between tangent spaces $T_x  \mathcal{M}$ and $T_{x_0} \mathcal{M}$, at $x $ and $x_0$, encodes an approximation of curvature data at $x_0$. In the context of data science, these tangent spaces are estimated using a point cloud of data.}
    \end{subfigure} 
    \hfill
 \begin{subfigure}[b]{0.5\textwidth}
 \centering
    \includegraphics[width=\textwidth]{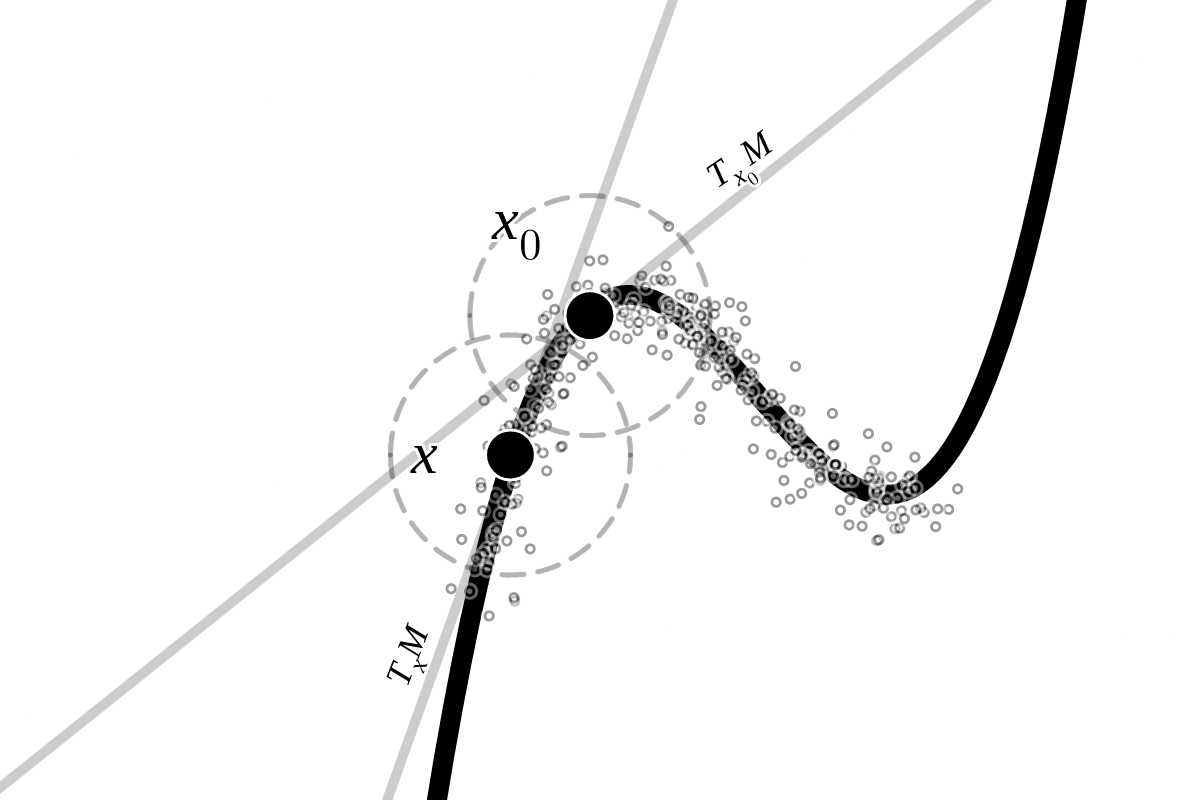}
\caption{Tangent spaces of two nearby points on a general curve (manifold).}
\end{subfigure}
    \caption{Visualizations of noisy manifolds. }
     \label{fig:curve_tangents}
\end{figure}

Our goal is to  estimate the absolute variation curvature from a noisy point cloud. Proceeding with our first computational experiment, Figure~\ref{fig:naive_bias} reveals an inherent bias when computing the naive curvature from a noisy point cloud.  In the underlying numerical experiment, we construct multiple noisy point clouds by uniformly sampling points from $n$-dimensional sphere of of radius $r>0$ embedding in $\mathbb{R}^{n}$, $S_r^{n-1} \subset \mathbb{R}^n$, and add ambient $n$-dimensional Gaussian noise to each data point independently. For each point cloud constructed in this way, we find estimates for the absolute variation curvature using tangent space approximations (as described above). Figure ~\ref{fig:naive_bias} shows the histogram of these calculated curvature values. As we increase the dimension, we observe a stronger bias in the observed values of $\omega$ such that even the mean (or mode) of the calculated curvature values is not a good representative of the true curvature in higher dimensions. The same is true if we increase the noise level (cf. Figure 
in the SI and Figure~\ref{fig:cur_noisy_r=2}). In Section ~\ref{sec:prob_decoding_avc} we incorporate a probabilistic framework to understand the effect noisy point cloud in tangent space estimations, and consequently, on the estimated curvature and use that  in decoded curvature estimation in order to remediate the bias. 

\begin{figure}[H]
    \centering
    \includegraphics[width=0.6\linewidth]{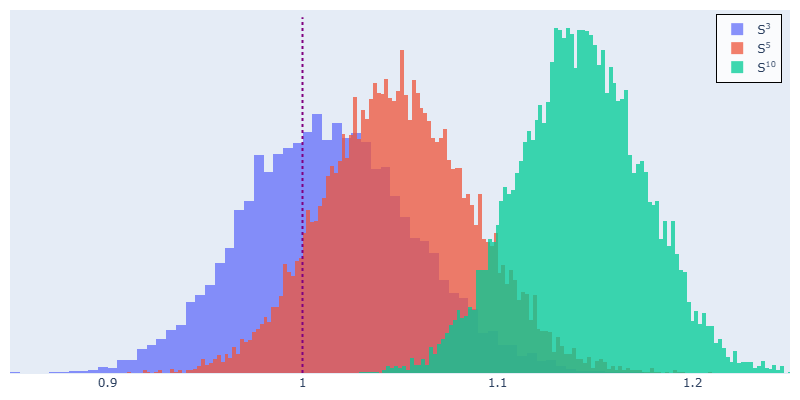}
    \caption{Histograms of curvature estimates computed from 50,000 random samples of points on spheres $S^3$, $S^5$, and $S^{10}$ with radius 1, using parameters shown in Table S1 (in the SI). The plots illustrate how bias in the naive calculation of absolute variation curvature increases with dimension. The dash line represents the true value of curvature for all cases.}
    \label{fig:naive_bias}
\end{figure}

\section{Probabilistic Decoding of Absolute Variation Curvature - Case of Codimension One}\label{sec:prob_decoding_avc}

Recall that absolute variation curvature is defined in terms of tangent spaces (points in Grassmanians), and requires that we estimate tangent spaces near a point $x_0\in \mathcal{M}$ (see Section~\ref{sec:prob_calib}). We focus here on  the case of codimension one, in which these tangent spaces are represented by unit normal vectors in $S^m$. The general situation is approachable, but there are subtleties that come with codimension greater than one. We discuss those in Section~\ref{sec:Discussion}.  

In the setting of codimension one the change in (noisy) tangent space approximation is thus measured in terms of the angle between normal vectors. With that, the general schema presented in Section~\ref{sec:prob_calib} to decode the absolute variation curvature is broken up into two parts: 

\begin{enumerate}
    \item \textbf{Calculation of the pushforward for the noisy angle computation}: Consider two tangent spaces with angle $\alpha$ relative to each other. I.e. $\alpha$ is the angle between the corresponding normal vectors. If we add vMF noise to the normal vectors, what is the distribution of the estimated angle? How does this distribution represent the true $\alpha$? 

    \item \textbf{Calculation of the pushforward for noisy curvature computation}: For this, we ask the question, if given a point on $\mathcal{M}$ with known (local) absolute variation curvature (e.g., the sphere $S_r^m$ of radius $r$ embedded in $\mathbb{R}^{m+1}$ in the usual way as a model space). If we add vMF noise to the normal vectors and compute the discrete approximation of absolute variation curvature, what is the distribution of this resulting value? How does this distribution represent the true absolute variation curvature? 

\end{enumerate}


Let $X$ be the random variable representing the tangent space estimation at $x\in \mathcal{M}$, and let $\Theta$ be a random variable associated with the estimated angle between two nearby tangent spaces. Note that in our context (codimension one)  $X$ is a distribution on $S^m$. The distribution of 
$X$ depends on the density of point cloud together with intensity of ambient noise. The distribution of $\Theta$ on the other hand depends on the distribution of $X$ together with the local distribution of curvature of $\mathcal{M}$. 

Initially, we assume that 
$X$ is distributed according to a  von Mises-Fisher distribution (vMF) with the density 
    \[     f_X(x) = c(\kappa) \exp\left( \kappa  \mu^T x  \right), \qquad c(\kappa) = \frac{\kappa^{(m+1)/2-1}}{(2\pi)^{(m+1)/2} I_{(m+1)/2-1}(\kappa)} \] 
where $\mu$ and $\kappa$ are the mean and concentration parameters\footnote{See https://en.wikipedia.org/wiki/Von\_Mises-Fisher\_distribution.} respectively, and $I_k$ is modified Bessel function of the first kind. As we will see, the we will be able to use a mixture of vMF distributions to incorporate more general noise models. The following theorem shows how a vMF noise in tangent space estimation transforms to a noisy estimation of angle between tangent spaces.

\begin{theorem}\label{thm:anglepushforward}
(Noise pushforward for angle computation) Let $\mu_0,\mu \in S^m$, and let $\alpha$ be the angle between $\mu_0$ and $\mu$. Let $X$ be a $S^m$-valued random variable distributed according to the  von Mises-Fisher distribution with (Frechet) mean of $\mu$ and concentration parameter $\kappa>0$.  Then, the random variable $\Theta$ defined as the angle between $\mu_0$ and $X$ has the probability density
  $$
        f_\Theta(\theta) = 2\sin^{m-1}(\theta) \cosh{(\kappa\cos(\alpha)\cos(\theta))} \frac{I_{\frac{m-2}{2}}(\kappa \sin(\alpha) \sin(\theta))}{(\kappa \sin(\alpha) \sin(\theta))^{\frac{m-2}{2}}} \cdot \frac{\kappa^{\frac{m-1}{2}}}{\sqrt{2\pi} \cdot I_{\frac{m-1}{2}}(\kappa)}
    $$
     where $\theta \in [0,\pi]$ and $I_k$ is the modified Bessel function of the first kind.
    
\end{theorem}

\begin{proof}
    First, we define a coordinate system on $S^m \subset \mathbb{R}^{m+1}$ relative to the unit vector $\mu_0$. Let $e_1, \dots, e_{m+1}$ be an orthonormal basis for $\mathbb{R}^{m+1}$ such that $\mu_0 = e_{m+1}$. This leads to a coordinate representation of $S^m$ with $x=(x_1, \dots, x_{m+1})=x_1 e_1 + \dots + x_{m+1}e_{m+1}$ where $x_j = \langle x, e_j \rangle$ and $\sum_j x_j^2 = 1 $. Let $\phi_0 \in [0,\pi]$ the angle between $x\in S^m$ and $e_{m+1}$ (i.e. $\mu_0$) and extend this to an spherical coordinate on $S^m$ by introducing $\phi_1, \dots, \phi_{m-1}$ relative to above orthonormal basis. We write the coordinate transformation map as $(\phi_0, \dots, \phi_{m-1}) = \Phi(x_1, \dots, x_{m+1})$ and $ = (x_1, \dots, x_{m+1})=\Phi(\phi_0, \dots, \phi_{m-1})$. Given such coordinate system we have $x_{m+1} = \arccos{\phi_0}$, and the relation between the two coordinate systems is given by 
    \begin{align*}
        x_{m+1} =& \, \cos(\phi_0)\\
        x_m =&  \,\sin(\phi_0) \cos(\phi_1) \\
        x_{m-1} =&  \,\sin(\phi_0) \sin(\phi_1) \cos(\phi_2) \\
        &\vdots \\
        x_2 = & \, \sin(\phi_0) \dots \sin(\phi_{m-2}) \cos(\phi_{m-1}) \\ 
        x_1 = & \, \sin(\phi_0) \dots \sin(\phi_{m-2}) \sin(\phi_{m-1}) 
    \end{align*} 
Note that $\phi_0, \dots, \phi_{m-2}$ all take values in $[0,\pi]$ while $\phi_{m-1}$ takes value in $[0,2\pi)$. The vector $\mu$ assumed to have angle $\theta$ with $\mu_0$, and hence by a symmetry argument we can assume that $\mu$ is given by $\phi_0 = \theta, \phi_1=0, \dots \phi_{m-1}=0$ and in Euclidean coordinates
\begin{equation}\label{eq:murep}
    \mu=(0,\dots, 0, \sin(\alpha), \cos(\alpha))^T
\end{equation}
   The random estimate of the angle $\theta$ is given by $\Theta = \cos^{-1}(\left| X\cdot \mu_0 \right|) \in [0,\pi)$ but we begin by finding a probability density for $\hat{\Theta} = \cos^{-1}( X\cdot \mu_0 ) \in [0,2\pi)$ where $X$ follows a  von Mises-Fisher (vMF) distribution. To achieve this, we transform the vMF probability density 
    $$
    f_X(x) = c(\kappa) e^{\kappa \langle \mu , x \rangle}
    $$
    into our spherical coordinate. Since $\hat{\Theta}$ only depends on $\phi_0$ coordinate, the probability density $f_{\hat{\Theta}}$ in the spherical coordinate is calculated via the marginals 
     \begin{equation}
            f_{\Tilde{\Theta}}(\theta) = \int_0^{2\pi} \int_0^{\pi} \cdots \int_0^{\pi} f_X\left(\Phi^{-1}(\theta, \phi_1, \dots, \phi_{m-1})\right) \; \left| \det J_{\Phi}(\theta, \phi_1, \dots, \phi_{m-1})\right| \, d\phi_1 \cdots d\phi_{m-1}
            \label{eq:theta_density_1}
        \end{equation}
    where $J$ denotes the Jacobian associated with spherical coordinate systems, 
       \begin{equation}
            \Big| \det(J_{\Phi}(\theta, \phi_1, \dots, \phi_{m-1}))\Big| = \sin^{m-1}(\theta)\cdot\sin^{m-2}(\phi_1) \cdots \sin(\phi_{m-2})
            \label{eq:jacobian}
        \end{equation}
    Using (\ref{eq:murep}), $\langle \mu, x \rangle$ is expressed as $\left(\cos(\alpha) \cos(\theta) + \sin(\alpha) \sin(\theta) \cos(\phi_1) \right)$ in spherical coordinate, and from \eqref{eq:jacobian} and\eqref{eq:theta_density_1} we have

    \begin{align}
        f_{\Tilde{\Theta}}(\theta)
        &= \int_0^{2\pi} \int_0^{\pi} \cdots \int_0^{\pi} c(\kappa) \sin^{m-1}(\theta)\cdot\sin^{m-2}(\phi_1) \cdots \sin(\phi_{m-2}) \notag\\
        &\qquad \cdot \exp{\left(\kappa(\cos(\alpha) \cos(\theta) + \sin(\alpha) \sin(\theta) \cos(\phi_1)) \right)} d\phi_1 \cdots d\phi_{m-1} \notag\\
        &= c(\kappa) \sin^{m-1}(\theta) \exp{(\kappa\cos(\alpha)\cos(\theta))} \notag\\
        & \qquad \cdot \int_0^{2\pi} \int_0^{\pi} \cdots \int_0^{\pi} \sin^{m-2}(\phi_1) \cdots \sin(\phi_{m-2})\exp{(\kappa \sin(\alpha) \sin(\theta) \cos(\phi_1))} d\phi_1 \cdots d\phi_{m-1} \notag\\
        &= c(\kappa) \sin^{m-1}(\theta) \exp{(\kappa\cos(\alpha)\cos(\theta))} \notag\\
        & \qquad \cdot \left( \int_0^{\pi} \sin^{m-2}(\phi_1)\exp{(\kappa \sin(\alpha) \sin(\theta) \cos(\phi_1))} d\phi_1\right)\notag\\
        & \qquad \cdot \left(\int_0^{\pi} \sin^{m-3}(\phi_2) d\phi_2\right) \cdots \left(\int_0^{\pi} \sin(\phi_{m-2}) d\phi_{m-2}\right) \left(\int_0^{2\pi} d\phi_{m-1}\right) \label{eq:theta_density_2}
    \end{align}
    To evaluate these integrals we use the following known integration formulas    
   \begin{equation}
       \int_0^{\pi} \sin^{m-1-j}(\phi_{j}) d\phi_{j} = \frac{\sqrt{\pi} \Gamma(\frac{m-j}{2})}{\Gamma(\frac{m-j+1}{2})}
   \end{equation} 

   \begin{equation}
       I_v(z) = \frac{(\frac{1}{2}z)^v}{\sqrt{\pi} \Gamma(v+\frac{1}{2})} \int_0^\pi \exp(\pm z\cos(\theta))(\sin(\theta))^{2v} d\theta
   \end{equation}
    leading to 
    $$\prod_{j=2}^{m-2} \int_0^{\pi} \sin^{m-1-j}(\phi_{j}) d\phi_{j} = \prod_{j=2}^{m-2} \frac{\sqrt{\pi} \Gamma(\frac{m-j}{2})}{\Gamma(\frac{m-j+1}{2})}= \pi^{\frac{m-3}{2}} \frac{1}{\Gamma(\frac{m-1}{2})}$$
    and 
    \begin{align*}
        & \qquad \int_0^{\pi} \sin^{m-2}(\phi_1)\exp{(\kappa \sin(\alpha) \sin(\theta) \cos(\phi_1))} d\phi_1\\
        &= \frac{\sqrt{\pi} \Gamma( \frac{m-1}{2})}{(\frac{1}{2})^{\frac{m-2}{2}}(\kappa \sin(\alpha) \sin(\theta))^{\frac{m-2}{2}}} \cdot I_{\frac{m-2}{2}}(\kappa \sin(\alpha) \sin(\theta))
    \end{align*}
 respectively. Finally, putting these together, we obtain
    \begin{align}
        f_{\Tilde{\Theta}}(\theta) &= c(\kappa) \sin^{m-1}(\theta) \exp{(\kappa\cos(\alpha)\cos(\theta))} \notag \\
        &\qquad \cdot \frac{\sqrt{\pi} \Gamma( \frac{m-1}{2})}{(\frac{1}{2})^{\frac{m-2}{2}}(\kappa \sin(\alpha) \sin(\theta))^{\frac{m-2}{2}}} \cdot I_{\frac{m-2}{2}}(\kappa \sin(\alpha) \sin(\theta)) \cdot  \pi^{\frac{m-3}{2}} \frac{1}{\Gamma(\frac{m-1}{2})} \cdot (2\pi) \notag \\
        &= \sin^{m-1}(\theta) \exp{(\kappa\cos(\alpha)\cos(\theta))} \frac{I_{\frac{m-2}{2}}(\kappa \sin(\alpha) \sin(\theta))}{(\kappa \sin(\alpha) \sin(\theta))^{\frac{m-2}{2}}} \cdot \frac{\kappa^{\frac{m-1}{2}}}{\sqrt{2\pi} \cdot I_{\frac{m-1}{2}}(\kappa)}
    \end{align}

    Now, to get $f_\Theta(\theta)$, note that
    \begin{equation}
        f_\Theta(\theta) = f_{\Tilde{\Theta}}(\theta) + f_{\Tilde{\Theta}}(\pi - \theta)
    \end{equation}

    and hence
    \begin{align*}
    f_\Theta(\theta) &= \sin^{m-1}(\theta) \exp{(\kappa\cos(\alpha)\cos(\theta))} \frac{I_{\frac{m-2}{2}}(\kappa \sin(\alpha) \sin(\theta))}{(\kappa \sin(\alpha) \sin(\theta))^{\frac{m-2}{2}}} \cdot \frac{\kappa^{\frac{m-1}{2}}}{\sqrt{2\pi} \cdot I_{\frac{m-1}{2}}(\kappa)}\\
    & \quad + \sin^{m-1}(\pi - \theta) \exp{(\kappa\cos(\alpha)\cos(\pi - \theta))} \frac{I_{\frac{m-2}{2}}(\kappa \sin(\alpha) \sin(\pi - \theta))}{(\kappa \sin(\alpha) \sin(\pi-\theta))^{\frac{m-2}{2}}} \cdot \frac{\kappa^{\frac{m-1}{2}}}{\sqrt{2\pi} \cdot I_{\frac{m-1}{2}}(\kappa)}\\
    &= \sin^{m-1}(\theta) \exp{(\kappa\cos(\alpha)\cos(\theta))} \frac{I_{\frac{m-2}{2}}(\kappa \sin(\alpha) \sin(\theta))}{(\kappa \sin(\alpha) \sin(\theta))^{\frac{m-2}{2}}} \cdot \frac{\kappa^{\frac{m-1}{2}}}{\sqrt{2\pi} \cdot I_{\frac{m-1}{2}}(\kappa)}\\
    & \quad + \sin^{m-1}(\theta) \exp{(-\kappa\cos(\alpha)\cos(\theta))} \frac{I_{\frac{m-2}{2}}(\kappa \sin(\alpha) \sin(\theta))}{(\kappa \sin(\alpha) \sin(\theta))^{\frac{m-2}{2}}} \cdot \frac{\kappa^{\frac{m-1}{2}}}{\sqrt{2\pi} \cdot I_{\frac{m-1}{2}}(\kappa)}\\
    & = 2\sin^{m-1}(\theta) \cosh{(\kappa\cos(\alpha)\cos(\theta))} \frac{I_{\frac{m-2}{2}}(\kappa \sin(\alpha) \sin(\theta))}{(\kappa \sin(\alpha) \sin(\theta))^{\frac{m-2}{2}}} \cdot \frac{\kappa^{\frac{m-1}{2}}}{\sqrt{2\pi} \cdot I_{\frac{m-1}{2}}(\kappa)}\\
\end{align*}

\end{proof}


In Theorem~\ref{thm:curv_pushforward} we consider the case of curvature estimation for $\mathcal{M} = S_r^m$ of radius $r$ in $\mathbb{R}^{m+1}$. We start with random points on $\mathcal{M}$. The proof can be found in the SI.  As we will see in Section~\ref{sec:experiments} this will help us address the bias that would occur in a naive absolute variation curvature estimate in the case of adding noise to the randomly points.

    \begin{theorem}\label{thm:curv_pushforward}
 (Noise pushforward for curvature computation) Fix a point $x$ on $S_r^m$ together with a collection of points $x_1, \dots, x_N \in S_r^m$ with a constant distance $\epsilon>0$ to $x$. Let $X_j$, $j=1,\dots, N$ be a $S^m$-valued random variable associated with the tangent space estimation at $x_j$, distributed according to the von Mises-Fisher distribution with mean of $\mu$ and concentration parameter $\kappa$. An estimation of absolute variation curvature using a discretization of formula (\ref{eq:absvc}) is expressed as a random variable as 
 $$ \frac{1}{N} \sum_{j=1}^N \Omega_{x_j}(x) = \frac{1}{N} \sum_{j=1}^N  \frac{2 \sin (\Theta(x,x_j)/2)}{\| x-x_j \|^2}$$
 and the probability distribution of $\Omega_{x_j}(x)$ is given by 
 \begin{multline}
    f_\Omega(\omega ) = \frac{\kappa^{\frac{m-1}{2}}}{\sqrt{2\pi} \cdot I_{\frac{m-1}{2}}(\kappa)} \cdot \left(\frac{2}{C}\right)^{m} \cdot \left( 1-\frac{\omega^2}{C^2}\right)^{\frac{m-2}{2}} \cdot \omega^{m-1} \\ \cdot \cosh{\left(\kappa\cos(\alpha)(1-\frac{2\omega^2}{C^2})\right)} \cdot \frac{I_{\frac{m-2}{2}}\left(\kappa \sin(\alpha)\frac{2\omega}{C} \sqrt{1-\frac{\omega^2}{C^2}}\right)}{(\kappa \sin(\alpha) \frac{2\omega}{C} \sqrt{1-\frac{\omega^2}{C^2}})^{\frac{m-2}{2}}}
    \end{multline}
    where $C = \frac{1}{r}\sqrt{\frac{2}{1 - \cos(\alpha)}}$ and $I_k$ is the modified Bessel function of the first kind.
\end{theorem}

\begin{proof}
    Let $$g(\theta) =  \frac{2\sin(\frac{\theta}{2})}{||x_j-x||_{\mathbb{R}^{m+1}}} =  C \sin(\frac{\theta}{2}).$$
If $\omega = g(\theta)= C \sin(\theta)$ then $\theta = g^{-1}(\omega) = 2 \sin^{-1}(\omega / C)$ and 
     \begin{equation}
        \frac{d g^{-1}}{d\omega} = \frac{2}{C\sqrt{1-\frac{\omega^2}{C^2}}}
        \label{eq:inverse_cur_dev}
    \end{equation}
    the probability distribution of $\Omega_{x_j}(x)$ as a transformation of the random variable $\Theta(x,x_i)$ (which itself was a transformation of $X_j$) is given by 
\begin{align}
        f_\Omega(\omega) &= f_\Theta(2 \sin^{-1}(\frac{\omega}{C})) \cdot \left|\frac{2}{C \sqrt{1-\frac{\omega^2}{C^2}}}\right| \notag \\
        &= \frac{\kappa^{\frac{k-1}{2}}}{\sqrt{2\pi} \cdot I_{\frac{k-1}{2}}(\kappa)} \cdot \sin^{k-1}(h(\omega)) \cosh{(\kappa\cos(\alpha)\cos(h(\omega)))}  \notag \\
        &\qquad \cdot \frac{I_{\frac{k-2}{2}}(\kappa \sin(\alpha) \sin(h(\omega)))}{(\kappa \sin(\alpha) \sin(h(\omega)))^{\frac{k-2}{2}}} \cdot \frac{2}{C\sqrt{1-\frac{\omega^2}{C^2}}} \notag\\
        & = \frac{\kappa^{\frac{k-1}{2}}}{\sqrt{2\pi} \cdot I_{\frac{k-1}{2}}(\kappa)} \cdot \left(\frac{2}{C}\right)^{k} \cdot \left( 1-\frac{\omega^2}{C^2}\right)^{\frac{k-2}{2}} \cdot \omega^{k-1} \notag \\
        & \qquad \cdot \cosh{\left(\kappa\cos(\alpha)(1-\frac{2\omega^2}{C^2})\right)} \cdot \frac{I_{\frac{k-2}{2}}\left(\kappa \sin(\alpha)\frac{2\omega}{C} \sqrt{1-\frac{\omega^2}{C^2}}\right)}{(\kappa \sin(\alpha) \frac{2\omega}{C} \sqrt{1-\frac{\omega^2}{C^2}})^{\frac{k-2}{2}}}.
    \end{align}
\end{proof}

Theorem~\ref{thm:anglepushforward} produces a distribution of angles between the north pole and a random unit vector chosen according to vMF. This is  the distribution of random angles between tangent spaces, where one of them is the tangent space at the point of interest $x_0$. Given any angle, we get an estimate of the curvature at $x_0$. Thus, given the distribution of angles we get a distribution of curvature estimates. In Figure~\ref{fig:theoretical_curv}, we see how those distributions change with the dimension (denoted as $k$) and how the bias in such estimates increases with dimension (e.g., if you estimated the absolute variation curvature, the so-called ``naive curvature calculation"). It is useful to compare this with the empirical densities of Figure~\ref{fig:naive_bias}, confirming that such bias exists and increases with dimension. In the Section ~\ref{sec:experiments} we see this in our numerical experiments. Using Theorem ~\ref{thm:linearpushfwd} we will further see how we can use the pushforward to correct this bias. In Theorem ~\ref{thm:curv_pushforward}, the random variables $\Omega_{x_j}(x)$, $j=1, \dots, N$ are independent and identically distributed. Hence, estimated curvature values at $x_j, j=1, \dots, N$ leads to a sample of size $N$ from $\Omega$. The idea is to incorporate these samples to estimate the parameters of $\Omega$ encoding the true absolute variation curvature value.

\begin{figure}[H]
    \centering
    \includegraphics[width=0.6\linewidth]{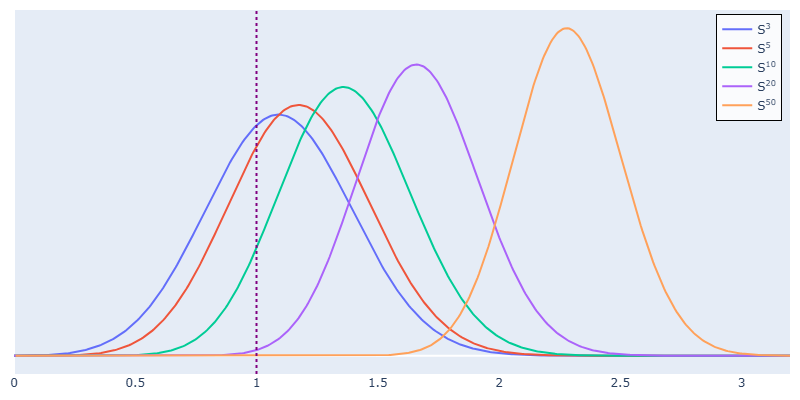}
    \caption{Theoretical probability density of naive curvature for spheres of dimension $3,5,10,20,50$. This confirms the empirical observations in Figure \ref{fig:naive_bias}.}
\end{figure}    \label{fig:theoretical_curv}

\subsection{General Noise: Mixtures of vMF}

So far, we see that if we assume a density function of  vMF for the noisy tangent space estimation, we can  construct  an explicit formula for the pushforward of the absolute variation curvature calculation. Finding such an explicit pushforward formula for some of the other probability distributions on the sphere (i.e., the codimension one case) can be addressed, using the fact that any continuous probability density function on $S^m$ can be approximated arbitrarily well (in the uniform norm) by a finite mixture of vMF distributions (see \cite{ng2020universalapproximationhypersphere}) and the linearity of our pushforward calculation. These facts are important enough that we frame them as theorems. 

\begin{theorem}\label{thm:linearpushfwd}
    Let $Y$ be the random unit vector associated to the estimated tangent space at $\mu \in S_r^m$. Asssume that $Y$ is a mixture of vMF random variables with probability density $f_Y = \sum a_j f_{\mu_j,\kappa_j}$ where $f_{\mu_j, \kappa_j}$ is the probability density of a vMF distribution with mean $\mu_j$ and scaling parameter $\kappa_j$. Then
    \begin{enumerate}
        \item The random variable $\Theta$ has the density  $ f_\Theta(\theta) = \sum_j a_j f_{\Theta_j}(\theta)$ where $ f_{\Theta_j}(\theta)$ is the distribution associated with $\mu=\mu_j$ and $\kappa=\kappa_j$ given in part 1 of Theorem \ref{thm:curv_pushforward} . 

    \item The random variable  $\Omega$ , representing curvature, has the density function $f_\Omega(\omega ) = \sum_j a_j  f_{\Omega_j}(\omega ) $ where $f_{\Omega_j}(\omega )$ is the distribution associated with $\mu=\mu_j$ and $\kappa=\kappa_j$ given in part 2 of Theorem \ref{thm:curv_pushforward}.
    \end{enumerate}
   
\end{theorem}

\begin{proof}
Given probability measures $\nu_1, \dots, \nu_m$  on $S^n$ and a map $\mathcal{J}: S^n \rightarrow \mathbb{R} $ the pushforward of the mixture $\sum a_j \nu_j$ is given by 
$$
\mathcal{J_*}\left( \sum_j a_j \nu_j \right) = \sum_j a_j \mathcal{J}_* \nu_j
$$
    using the linearity of pushforward operator. Both statemnent follow from this linearity. 
\end{proof}

We use the above statement to provide an accurate representation of the noise in tangent space estimation and as a result we obtain an accurate estimation of curvature. In our setup, we only use mixtures of the form $\sum_j a_j f_{\mu, \kappa_j}$ with fixed mean $\mu$. 



And thus, more generally:

\begin{theorem}\label{thm:arb_mixture}
    Let $Y$ be a random unit vector associated with the estimated tangent space at $\mu \in S_r^m$ with an arbitrary probability distribution function. Then the probability distributions of the corresponding $\Theta$ and $\Omega$ can be estimated by mixtures stated in Theorem \ref{thm:linearpushfwd}. 
\end{theorem}

\subsection{General manifolds and real-world datasets}

Theorem \ref{thm:curv_pushforward} and its extensions via mixture models, in their current form, apply to the case $\mathcal{M} = S^m\subset {\mathbb{R}^{m+1}}$. However, in the more general setting of a manifold of codimension one, we still face the challenge of noisy estimation of normal vectors (or tangent spaces) and their impact on curvature computation. There are several possible ways to adapt the framework of this paper to such a setting. First, observe that Theorem \ref{thm:curv_pushforward} applies directly to pairwise curvature estimation, $\Omega_{x_j}(x)$, for any fixed neighbor $x_j$ (see Figure~\ref{fig:curve_tangents}b). This naturally leads to a notion of directional curvature for an arbitrary manifold of codimension one:  the curvature at $x$ in the direction of $x_j$.

Another possibility (for a non-sphere manifold) is to compute pairwise curvature for each $x_j$, $j = 1, \dots, N$, and then apply the statistical framework described in this paper to estimate curvature values. Although our theorem relies on certain assumptions specific to the sphere, it can still mitigate the emerging bias in curvature estimation as the dimension increases.\\

\textbf{Real-world datasets. } We conclude by discussing how the proposed theoretical framework can potentially be applied to general real-world datasets, assuming the existence of analogues of Theorem \ref{thm:curv_pushforward} in arbitrary codimensions. Given a dataset and having no prior knowledge of $\mathcal{M}$, we propose a curvature estimation in four phases. We assume use of existing methods to estimate the local dimension $m < n$ of $\mathcal{M}$ near a point of interest $x \in \mathcal{M}$. We plan to implement this computational technique in future works.

\textit{Phase 1: Density and noise characterization}. Let $U$ be a local neighborhood of $x\in \mathcal{M}$ where the data points in $U$ are approximately lying on an estimated tangent space. The distance (squared) of these points from the tangent space characterizes the variance , $\sigma^2$, of ambient noise (a larger value of this variance translates into a larger deviation of the point cloud from $\mathcal{M} \mathbb{R}^n$). An estimation of the density parameter $\rho$ of the point cloud near $x$ is given by $\# \{\textit{points in } U \} /\operatorname{Vol(U)}$.


\textit{Phase 2: Bootstrapping tangent space samples}. Given the variance and density parameters, $\sigma^2$ and $\rho$, we can simulate point clouds near a $m$-dimensional linear space $L$ in $\mathbb{R}^n$. For any such simulated point cloud we can find an estimated tangent space relative to the true tangent $L$. Repeating this bootstrap process, we can obtain samples for the probability distributions associated with tangent space estimation.  An alternative approach is to provide a theoretical derivation of this probability distributions using $\sigma^2$ and $\rho$ directly.

\textit{Phase 3: Fitting to a mixture of vMF's}. Based on bootstrap samples in phase 2, we can employ a maximum likelihood method to estimate parameters of a mixture of vMF distributions that best represent the tangent space noise distribution.

\textit{Phase 4: Estimation of curvature}. This phase is performed as previously prescribed, using a push forward with respect to the mixture of vMF's obtained in phase 3.


\section{Numerical Experiments -- The case of $\mathcal{M} = S^m_r$}\label{sec:experiments}
\subsection{Framework}
Theorems~\ref{thm:linearpushfwd} and~\ref{thm:arb_mixture} allow  us in our numerical experiments to justify the modeling of the distribution of the estimated normal vectors using a finite mixture of vMF components.  This structure enables the use of maximum likelihood estimation (MLE) to infer the parameters required for our decoded curvature estimation. We proceed as follows. \\

 We use our method to estimate curvature on a sphere from a points randomly distributed around the north pole. We succeed here without mixtures. The only randomness comes from the sampling around the north pole. We first find that if we try to estimate the curvature from our formula directly (which depends on $\kappa$), then we obtain biased estimates, with a bias that increases with dimension. We then find that we can mitigate that bias through our derivation of the pushforward of the curvature distribution.  

We then bring more randomness into the situation: We consider a model in which we randomly sample points from the sphere (uniformly sampling every coordinate and then normalizing the vector and then fixing a neighborhood about the north pole and then adding $n-$dimensional Gaussian noise to each point except the north pole). Again, we have a naive estimate of curvature (see Algorithm 2 in the SI) which we show -- empirically -- is biased. We then show how using a mixture of pushforwards -- requiring different $\kappa$, but with a fixed $\mu$ -- we can recover the true curvature. Note: it's not a priori the case that we could mitigate the bias with a small mixture of pushforwards, but we show that it is possible in several cases.

As we will see, through this process we are able to greatly mitigate the bias that is implicit in a direct (naive) calculation of the absolute variation curvature. As per our discussion in Section~\ref{sec:prob_decoding_avc}, these ideas would transfer to improvements in the estimates of other kinds of scalar-valued curvatures.

\subsection{Sampling}

We generate point clouds by uniformly sampling on hyperspheres $S^m_r$, of radius $r>0$. In the \textbf{perfect sampling} setting, points are drawn uniformly on the hypersphere. In the \textbf{noisy sampling} setting, we add $n$-dimensional i.i.d. Gaussian noise with standard deviation $\sigma$ to the uniformly sampled points. Specifically, we test 3 noise levels $\sigma \in \{0.01,0.02,0.05\}$ in noisy sampling experiments.

For each hypersphere $S^m_r$, the total number of points sampled is determined by the product of a chosen density of points and the surface area of $S^m_r$. Here, density refers to the average number of points per unit surface area. The chosen densities and resulting number of samples used in each case are summarized in Table S1 (in the SI). For instance, consider the case of $S^3_1$, whose surface area is $\frac{2\pi^2}{\Gamma(2)} = 2\pi^2 \approx 19.74$. With a chosen density of 50, the total number of points sampled is $\lceil 50 \times 19.74 \rceil = 986$. We observe that under perfect sampling, and having sufficient number of sample points, the empirical distribution of estimated normal vectors can be approximated by a single  von Mises–Fisher distribution. In a noisy sampling setup, we use higher mixtures of von Mises–Fisher distribution to model the tangent space estimation error.

\subsection{Hyperparameters}

 We assume the local dimension of the manifold is known in each case. Thus, for a dataset sampled from $S^m_r$, the local dimension is set to $m$. The neighborhood radius $\epsilon$ used for tangent space estimation is fixed at $\epsilon=0.52$ for cases with $r=1$, and $\epsilon=0.78$ for cases with $r=2$. The neighborhood radius $\epsilon'$ used for curvature estimation is chosen to be $\epsilon' = \epsilon + 0.2r$ for each $S^m_r$. \\

 With this we then compute the decoded curvature estimate (see Algorithm 3 in the SI). We use maximum likelihood estimation (MLE) 
 which requires estimates of the concentration parameters $\{\kappa_i\}$ for the  von Mises–Fisher distributions (for $S^m_r$) in our mixture model and their corresponding weights in the mixture model. These concentration parameters $\{\kappa_i\}$ and their weights in the mixture model depend solely on the noise level and neighborhood size and variance used for tangent space estimation (see Algorithm 1 in the SI), and are thus relatively stable with under the same parameter configuration ($m$, $r$, $\epsilon$, $\sigma$, density). Therefore, for each $S^m_r$ and each parameter configuration, we estimate $\{\kappa_i\}$ and the mixture weights empirically by conducting 50,000 simulations in which the tangent space is estimated at the north pole using point clouds sampled from $S^m_r$.\\

 In the \textbf{perfect sampling} case, the distribution of estimated normal vectors is well-approximated by a single vMF distribution, corresponding to a one-component mixture model. In the \textbf{noisy sampling} case, more complex distributions arise, and we use mixture models with up to four components in the most challenging scenarios.

\subsection{Results}

For each hypersphere $S^m_r$ (i.e., an $m$-sphere with radius $r$), the theoretical absolute variation curvature at any $x \in S^m_r$ is given by $\omega_x = \frac{1}{r}$. Due to the rotational symmetry of hyperspheres, curvature estimation at any point on a hypersphere can be reduced, via an isometric transformation, to curvature estimation at the north pole. Thus, all curvature estimates in our experiments are performed at the north pole without loss of generality.

We conducted numerical experiments on data generated from hyperspheres of varying dimensions and radii to evaluate the accuracy and robustness of our method under different noise levels. The results of estimated naive curvature and decoded curvature under the \textbf{perfect sampling} setting can be found in Figure S2 the SI. We observe that the naive curvature estimation is consistently biased, and the bias increases with the local dimension $m$ of the manifold. In contrast, our decoded curvature estimation remains unbiased across all tested cases. 

In Figure~\ref{fig:cur_noisy_r=2} we show the results in the \textbf{noisy sampling} setting for  $r=2$ (the case of $r=1$ can be found in Figure S3 in the SI). The noise levels added are $\sigma \in \{0.01,0.02,0.05\}$. We observed that the naive curvature estimation becomes increasingly biased as the local dimension or noise level increases. However, our decoded curvature estimation remains close to the ground truth curvature regardless. Although our decoded curvature estimates exhibit slight bias in some cases, they remain significantly closer to the theoretical value compared to the highly biased naive estimates. As discussed in a previous section, this residual bias arises from imperfect representation of the tangent space error distribution (see Figure S1 in the SI). \\

\begin{figure}[htbp]
    \centering
    \caption{Decoded Curvature (red) vs Naive Curvature (blue) for noisy point clouds near $S^m \subset \mathbb{R}^{m+1}$. Rows correspond to dimensions $m=3,5,10,12$, respectively; columns correspond to the noise levels $\sigma = 0.01,0.02,0.05$, respectively. The radius is $r=2$ in all cases. Generally, the decoded curvature is peaked around or very near the true curvature value while the naive estimates get progressively worse with increased dimension or increased noise.}
    \label{fig:cur_noisy_r=2}

    \begin{subfigure}{0.9\textwidth}
        \centering
        \legendentry{red}{Decoded Curvature \quad\quad\quad}
        \legendentry{blue}{Naive Curvature}
    \end{subfigure}

    \vspace{1em}
    
    \begin{subfigure}[b]{0.9\textwidth}
        \centering
        \includegraphics[width=\textwidth]{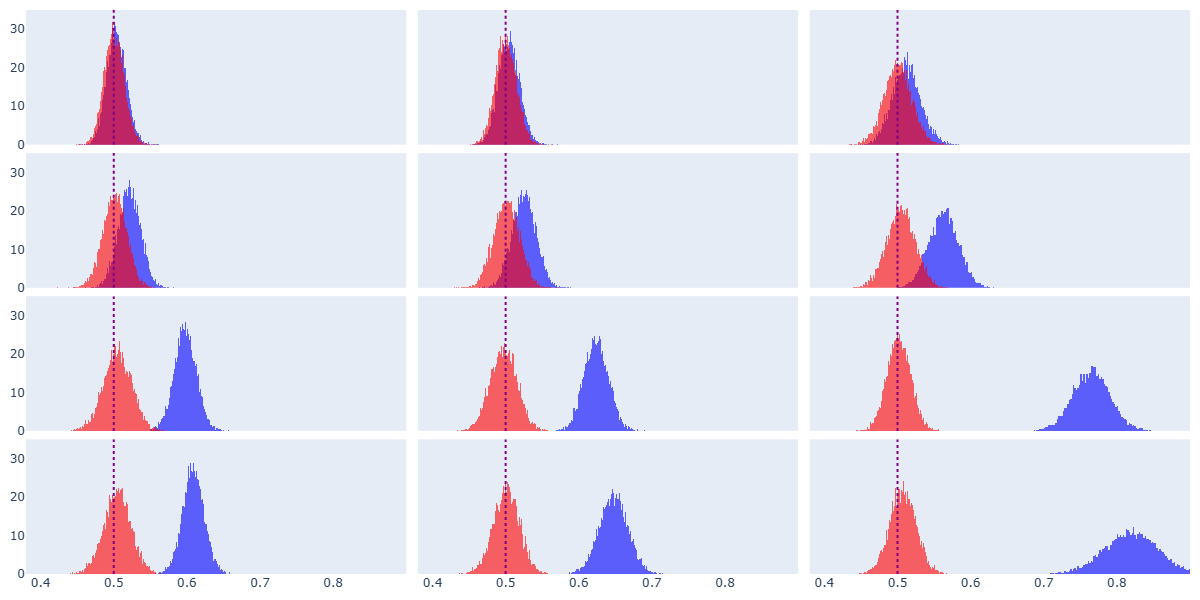}
    \end{subfigure}
\end{figure}

\clearpage

\newpage

\section{Discussion}\label{sec:Discussion}

Figure~\ref{fig:cur_noisy_r=2} as well as Figures S2 and S3 in the SI show several phenomena. First that in the naive calculation of curvature, bias increases both as noise increases and as dimension increases. The figures further show that the decoded curvature are generally highly localized around the actual curvature value, thus proving it to be a much better estimator in the noisy case. 

While as per Theorems~\ref{thm:linearpushfwd} and~\ref{thm:arb_mixture}, working with the vMF noise model is not a theoretical limitation, in practice it would be of interest to know if we directly estimate the tangent space estimation for a noisy point cloud. A natural model here is uniform sampling from the manifold with the addition of ambient Gaussian noise. 

All of our explicit calculations have been carried out in the simple case of codimension one which allows for a simplification of the tangent space representation (in this case, by the unit vector representing the orthogonal space). We reiterate that this work is applicable beyond the  reconstruction of a  sphere and the setting of manifolds of codimension one. First, note that for a submanifold of dimension $m$ in $\mathbb{R}^n$ the tangent spaces are represented by elements of the Grassmannian $\operatorname{Gr}(m,n)$, representing $m$-dimensional linear subspaces of $\mathbb{R}^n$. A coordinate system for $\operatorname{Gr}(m,n)$ is provided via the space of $n\times m$ matrices (two matrices represent the same element $\operatorname{Gr}(m,n)$ of they have the same column space). There is a matrix version of the von Mises-Fisher distribution with the density 
$$
f(x) = c(\kappa) \; e^{\kappa \operatorname{tr} \left( \mu^T x \right)}
$$
where $\kappa$ is the concentration parameter, $\mu$ and $x$ denote $n\times m$ matrices, representing elements of $\operatorname{Gr}(m,n)$. A Grassmannian-valued random variable in this setting provides a model for a noisy estimation of tangent spaces for a manifold of dimension $m$ embedded in $\mathbb{R}^n$. For a full extension of our work one needs to establish the analogs of the pushforward theorems for the curvature estimation (Theorems~\ref{thm:anglepushforward} and~\ref{thm:curv_pushforward}) by computing explicit transformation formulas for the Grassmannian-valued random variable. We postpone this question to a future project. 

Also, our explicit calculations are for absolute mean curvature. This methodology could be applied -- in principle -- to computing the probability density for  any entry in the second fundamental form given a noise model for estimating tangent spaces. Extension of our results to any one of these curvature-related scalars (or combinations thereof) would require the ability to compute the pushforward of the noisy estimate of the tangent spaces to entries of the second fundamental form. This too is an interesting future project. 




\clearpage

\newpage



\section*{Acknowledgments}
Mohammad Javad Latifi Jebelli was supported as a postdoctoral scholar from the U.S. Office of Naval Research under MURI grant N00014-19-1-242. We thank Soroush Vosoughi for helpful early discussions. 


\begin{appendices}

\section*{Shape Operator}\label{sec:shape_operator}
To each point $x \in \mathcal{M}$, we can associate a normal vector $\textbf{n}(x) \in \mathbb{R}^n$ such that $\textbf{n}(x) \perp T_x \mathcal{M}$. For a tangent vector $Y \in T_x \mathcal{M}$, the shape operator at $x \in\mathcal{M}$ is a linear map defined using the derivative of the normal vector field, $\textbf{n}$, in the direction of $Y$, i.e. 
$$\mathcal{L}_x Y = - D_{Y} \textbf{n}$$
The second fundamental form $\Pi$ is then a bilinear form on the tangent space $T_x \mathcal{M}$ 
$$
\Pi (Y , Z) = \langle \mathcal{L}_x(Y), Z \rangle
$$
the trace and determinant of $\mathcal{L}_x$ give the mean and the Gaussian curvature at $x$, respectively. Given a ($n$ by $n$) linear transformation $T$, we define $\psi(T)$ to be the mean value of  on $||Tx||$ for $x$ with $||x||=1$, more concretely 
$$
\psi(T) = \frac{1}{\operatorname{vol}(S^{n-1})} \int_{S^{n-1}} ||Tx|| \, dx
$$
it is not hard to see that $\psi(T)$ is a symmetric function of eigenvalues of $T$ and in the case of $T:\mathbb{R}^2 \rightarrow \mathbb{R}^2$, $\psi(T)$ can be expressed as an elliptic integral of second kind in $\lambda_1$ and $\lambda_2$. 

 Expressing the absolute variation curvature in terms of the eigenvalues $\lambda_1, \dots, \lambda_m$, the quantity $\psi({\mathcal{L}}_x)$ is positive and invariant under the transformation $\lambda_i \rightarrow -\lambda_i$.\\

We show that $\Omega(\cdot, \cdot)$ gives a discrete approximation of absolute variation curvature $\omega_x$.  We will focus on manifolds of codimension one. 

\begin{lemma}\label{prop:abs_mean_eq}
    Let $\mathcal{M}$ be a smooth manifold of codimension one in $\mathbb{R}^n$. The absolute variation curvature $\omega_x = \psi(\mathcal{L}_x)$ at $x \in \mathcal{M}$ can be expressed as the average of $\Omega_y(x)$ taken over all points $y$ on the boundary of an infinitesimal sphere $B_\epsilon(x) = \{x \in M: \|x-y\| = \epsilon \}$ centered at $x$. More precisely, 
    
    \begin{equation}\label{eq:mean_formula}      
    \psi(\mathcal{L}_x) = \lim_{\epsilon \rightarrow 0 } \frac{1}{|A_{\epsilon}|} \int_{A_{\epsilon}}\Omega_y(x)\, d\mu(y), \quad \quad A_{\epsilon} = \partial B_\epsilon(x)
    \end{equation}
    where $d\mu$ is the volume measure on $A_\epsilon$.
\end{lemma}
\begin{proof}
 For a small value of $\epsilon = d(x,y) >0$, $\theta (x,y)$ is the angle between normal vectors $\mathbf{n}(x)$ and $\mathbf{n}(y)$, at $x$ and $y$ respectively. Hence, $\theta (x,y) = \arccos(\mathbf{n}(x)\cdot \mathbf{n}(y))$ and 
 $$
 \Omega_y(x)= \frac{2 \sin (\arccos(\textbf{n}(x)\cdot \textbf{n}(y))/2)}{|| x-y ||}
 $$
on the other hand for two unit vectors $n_1$ and $n_2$ on $n$-dimensional sphere we have $|| n_1 - n_2 || = 2 \sin (\arccos(\frac{n_1\cdot n_2}{2}) $ implying that     
   $$
 \Omega_y(x)= \frac{|| \textbf{n}(x) - \textbf{n}(y)||}{|| x-y ||}
 $$ 
 Therefore, as $\epsilon \rightarrow 0$, $y\rightarrow x$ and $\Omega_y(x)$ converges to $||\mathcal{L}_x Y ||$. The right hand side of (\ref{eq:mean_formula}) is then the average of $||\mathcal{L}_x Y ||$ over all unit vectors $Y.$
 \end{proof}

\section*{Algorithms}

Following the definition, our algorithm for computing the curvature of a point cloud includes two major steps:
\begin{enumerate}
    \item Estimate tangent space at each point;
    \item Estimate curvature using tangent spaces.
\end{enumerate}

\subsection*{Tangent Space Approximation} 
To estimate the tangent space at each point, we used local PCA. 

\begin{algorithm}[H]
\caption{Find Tangent space at x using PCA in local neighborhood of radius $\epsilon$}\label{alg:find_tangent}
\begin{algorithmic}
\State Assume the local dimension is m
\State Perform Nearest Neighbors: X $\gets$ neighbors of x with $\|\text{$X_i$-x}\| < \epsilon$

\For{$X_i$ in X}
    \State $X_i$ $\gets$ $X_i$-x
    \State $X_i$ $\gets$ $\sqrt{\exp(-5 \cdot (\| \text{$X_i$} \|^2 / \epsilon))}$
\EndFor

\State D $\gets$ diag(X)
\State B $\gets$ X$^T$ D

\State Perform SVD: U, S, V$^T$ $\gets$ svd(B)
\State $T_x \gets$ first m columns of U
\State \Return $T_x$
\end{algorithmic}
\end{algorithm}

\subsection*{Naive Curvature Estimation Algorithm}

Once the tangent spaces have been estimated at each relevant point, we compute the curvature at a given point $x$ using the procedure described in Algorithm~\ref{alg:find_cur}.

\begin{algorithm}[H]
\caption{Find Curvature for x in local neighborhood of radius $\epsilon'$}\label{alg:find_cur_naive}
\begin{algorithmic}
\State $T_x$ $\gets$ tangent space at x
\State Perform Nearest Neighbors: Y $\gets$ neighbors of x with $\epsilon < \|\text{$Y_i$-x}\| < \epsilon'$

\For{$Y_i$ in Y}
    \State $T_i$ $\gets$ estimated tangent space at $Y_i$
    \State $\omega_i$ $\gets$ $\frac{ 2 \sin(\theta(T_x,T_i)/2)}{\| T_x - T_i \|}$
\EndFor
\State $\Bar{\omega} \gets \text{mean}(\omega_i)$
\State \Return $\Bar{\omega}$
\end{algorithmic}
\end{algorithm}

To estimate curvature at $x$, we would like to ensure that neighboring points $Y_i$ are approximately equidistant from $x$ (and hence subtend approximately equal geodesic angles). In practice, we restrict attention to points within a thin spherical shell defined by $\epsilon < \|\text{$Y_i$-x}\| < \epsilon'$. Such annular selection improves the stability and geometric consistency of the curvature estimates.

\subsection*{Decoded Curvature Estimation Algorithm}

With the density function $f_\Theta$, we suggest a modified curvature computation algorithm for a more accurate estimation of curvature given that the manifold is $S^m_r \in \mathbb{R}^{m+1}$.

\begin{algorithm}[H]
\caption{Find Unbiased Curvature for x in local neighborhood of radius $\epsilon'$}\label{alg:find_cur}
\begin{algorithmic}
\State $T_x$ $\gets$ tangent space at x
\State Perform Nearest Neighbors: Y $\gets$ neighbors of x with $\epsilon < \|\text{$Y_i$-x}\| < \epsilon'$

\For{$Y_i$ in Y}
    \State $T_i$ $\gets$ estimated tangent space at $Y_i$
    \State $\theta_i$ $\gets$ $\theta(T_i,T_x)$
\EndFor
\State Perform MLE on $\{\theta_i\}$ to estimate $\alpha$ based on $f_\Theta$
\State $\Bar{d}$ $\gets$ mean($\|Y_i-x\|$)
\State $\Tilde{\omega} \gets \frac{2 \sin(\alpha/2)}{d}$
\State \Return $\Tilde{\omega}$
\end{algorithmic}
\end{algorithm}  




\newpage

\begin{figure}[h]
    \centering
    \begin{subfigure}{0.8\textwidth}
        \includegraphics[width=\linewidth]{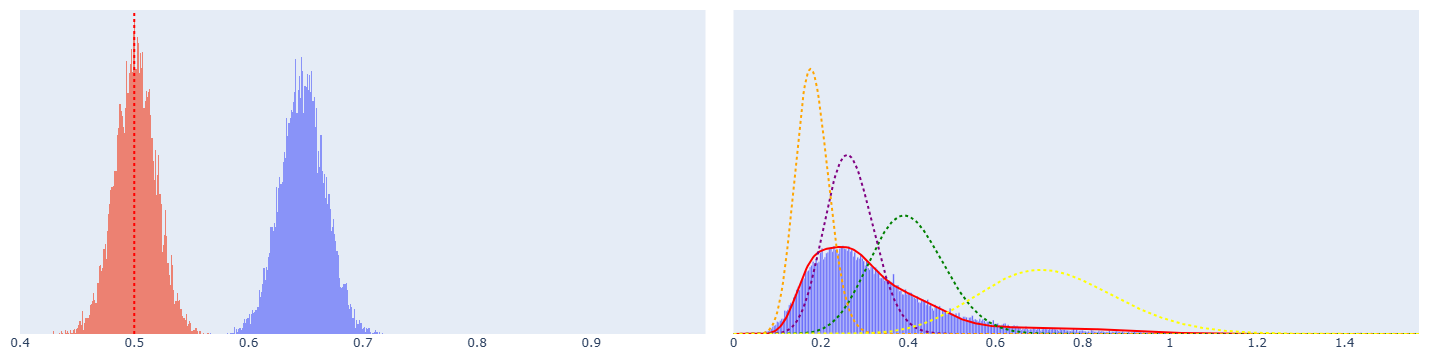}
        \caption{$S^{12}_2, \sigma = 0.02$}
    \end{subfigure}
    
    \begin{subfigure}{0.8\textwidth}

        \includegraphics[width=\linewidth]{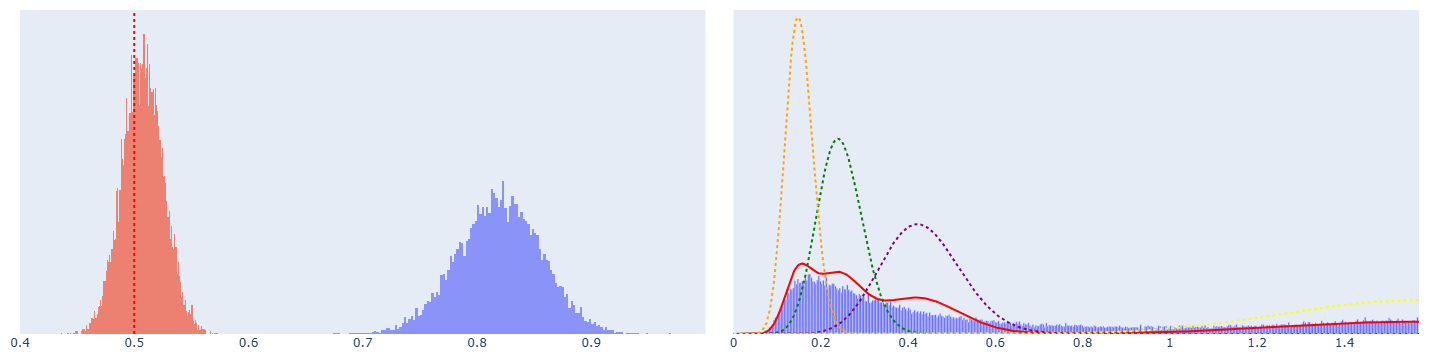}
        \caption{$S^{12}_2, \sigma = 0.05$}
    \end{subfigure}
    \caption{
        Left: decoded curvature estimates (red) and naive curvature estimates (blue) for $S^{12}$ and radius $2$, with noise levels $\sigma = 0.02, 0.05$. 
        Right: distribution of tangent space approximation errors, with empirical histogram (blue) and the MLE fit (red line) from a mixture of vMF components (dashed lines). 
        We see how in both of these high-dimensional cases the decoded estimate is strongly peaked around the correct curvature ($0.5$ for a sphere of radius $2$). The fits illustrate how the tangent space error distribution fit affects the bias in our decoded curvature estimation.
    }
\end{figure}

\begin{figure}[htbp]

    \label{fig:cur_perfect}
    
    \begin{subfigure}{0.9\textwidth}
        \begin{center}
            \legendentry{red}{Decoded Curvature \quad\quad\quad}
             \legendentry{blue}{Naive Curvature}
        \end{center}
    \end{subfigure}
    
    \centering

    \begin{subfigure}[b]{0.9\textwidth}
        \centering
        \includegraphics[width=\textwidth]{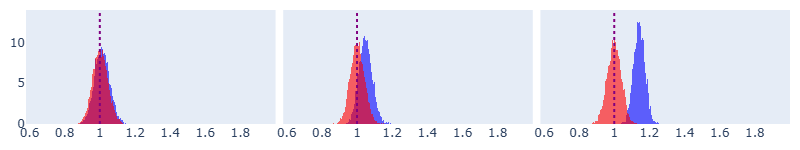}
        \caption{$r=1$}
    \end{subfigure}
    
    \begin{subfigure}[b]{0.9\textwidth}
        \centering
        \includegraphics[width=\textwidth]{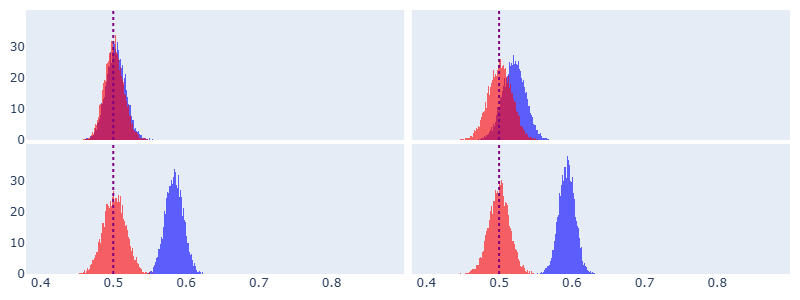}
        \caption{$r=2$}
    \end{subfigure}
    \caption{Decoded curvature versus naive curvature with perfect sampling. part (a) shows the results for $r=1$, and $m=3,5$, respectively; part (b) shows the results for $r=2$, and $m=3,5,10,12$, respectively.Generally, the decoded curvature is peaked around or very near the true curvature value while the naive estimates get progressively worse with increased dimension.}
\end{figure}

\begin{figure}[htbp]
    \centering
    \label{fig:cur_noisy_r=1}
    \begin{subfigure}{0.9\textwidth}
        \centering
        \legendentry{red}{Decoded Curvature \quad\quad\quad}
        \legendentry{blue}{Naive Curvature}
    \end{subfigure}

    \vspace{1em}
    \begin{subfigure}[b]{0.9\textwidth}
        \centering
        \includegraphics[width=\textwidth]{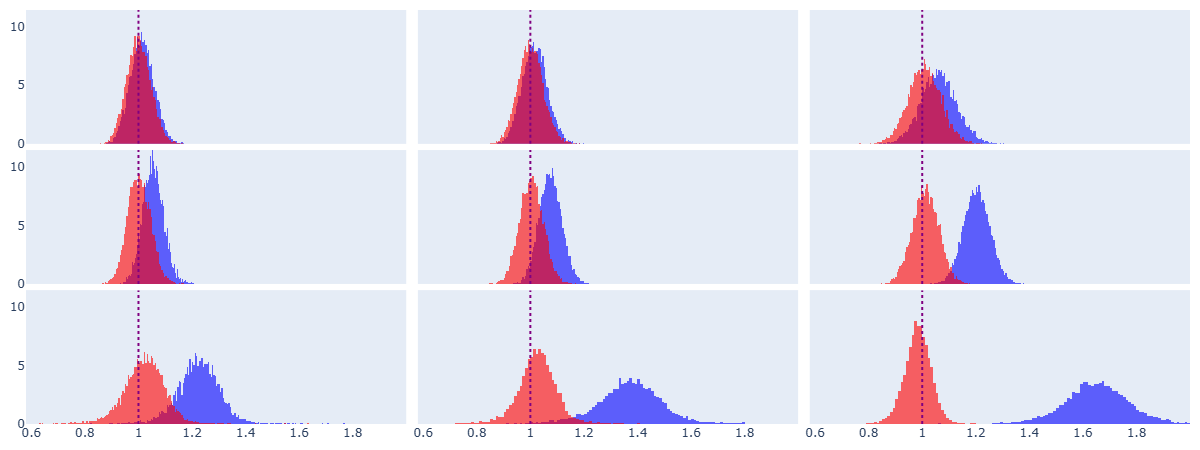}
    \end{subfigure}
    \caption{Decoded curvature versus naive curvature with noisy sampling for the case $r=1$. Rows correspond to $m=3,5$, respectively; columns correspond to $\sigma = 0.01, 0.02, 0.05$, respectively. Generally, the decoded curvature is peaked around or very near the true curvature value while the naive estimates get progressively worse with increased dimension or increased noise.}

\end{figure}

\newpage

\begin{table}[h!]
\centering
\caption{Densities and Number of Samples Used for Chosen Hyperspheres}
\begin{tabular}{ccc}
 
\textbf{Sphere} & \textbf{Density} & \textbf{Number of Samples} \\
\hline
$S^3_1$  & 50     & 1000 \\
$S^5_1$  & 200    & 6400 \\
$S^10_1$ & 100000 & 2,100,000\\
$S^3_2$  & 20     & 3,160 \\
$S^5_2$  & 25     & 24,825 \\
$S^{10}_2$ & 500   & 10,611,500 \\
$S^{12}_2$ & 3000   & 145,470,000 \\

\end{tabular}
\label{table:density_table}
\end{table}

\end{appendices}

\bibliographystyle{plain}
\bibliography{refs}

\end{document}